\documentclass[12pt]{article}

\usepackage{amsmath,amsthm,amsfonts,amssymb,amscd}

\allowdisplaybreaks[4]
\newtheorem {Lemma}{Lemma}
\newtheorem {Theorem} {Theorem}

\allowdisplaybreaks [4]
\begin{document}

\title{On the transmission of uniform unicyclic hypergraphs}

\author{Hongying Lin\footnote{ E-mail:
lhongying0908@126.com}, Bo Zhou\footnote{Corresponding author. E-mail:
zhoubo@scnu.edu.cn}\\[1.5mm]
School of Mathematical Sciences, South China Normal University, \\
Guangzhou 510631, P. R. China }

\date{}
\maketitle

\begin{abstract}  The transmission of a connected hypergraph is defined as the summation of distances between all unordered pairs of distinct vertices.
We determine the unique uniform unicyclic hypergraphs of fixed size with minimum and maximum transmissions, respectively.
\\ \\
{\it  AMS classifications:}  05C65, 05C35\\ \\
{\it  Key words:} uniform hypergraph, unicyclic hypergraph, transmission, extremal hypergraph
\end{abstract}

\section{Introduction}

A hypergraph $G$  is a pair $(V, E)$, where $V=V(G)$ is a nonempty finite set called the vertex set of $G$ and $E=E(G)$
is a family of subsets of $V(G)$ called the edge set of $G$ \cite{Be}.
The size of $G$ is the  cardinality of $E(G)$.
For an integer $k\ge 2$, a hypergraph is $k$-uniform if all its edges have cardinality $k$.
In this paper we only consider $k$-uniform hypergraphs where $k\ge 2$
is fixed.

A sequence of $v_0e_1v_1\dots e_kv_k$, where $v_i$ are vertices of $H$,  $e_i$ are its edges, and $v_{i-1}, v_i\in e_i$ for $i=1, \dots, k$, is a walk of length $k$ between $v_0$ and $v_k$. By a component of $H$ containing some vertex $v\in V(G)$, we mean a subhypergraph which consists of all vertices $v¡¯$ and edges $e¡¯$ belonging to some walk containing $v$. $H$ is connected if it has only one component.

For $u, v\in V(G)$, a walk  from $u$ to $v$ in $G$ is defined to be an alternating sequence of  vertices and edges
$(v_0,e_1,v_1,\dots,v_{p-1},e_p, v_p)$ with $v_0=u$ and $v_p=v$ such that
$v_{i-1}\ne v_i$ and
$v_{i-1}, v_i\in e_i$  for $i=1,\dots,p$. The value $p$ is the length of this walk.
A path is a walk with all $v_i$ distinct and all $e_i$ distinct. In particular, a vertex $u\in V(G)$ is viewed as a path (from $u$ to $u$) of length $0$.
A cycle is a walk containing at least
two edges, all $e_i$ are distinct and all $v_i$ are distinct except $v_0=v_p$.

If there is a path from $u$ to $v$ for any $u,v\in V(G)$, then we say that $G$ is connected.
A hypertree is a connected hypergraph with no cycles. A unicyclic hypergraph is a connected hypergraph with exactly one cycle. Note that a
$k$-uniform unicyclic hypergraph of size $m\ge 2$ always has $m(k-1)$ vertices, where $m\ge 3$ if $k=2$. 

Let $G$ be a connected  hypergraph. For $u,v\in V(G)$, the distance between $u$ and $v$ is the length of a shortest path from $u$ to $v$ in $G$,
denoted by $d_{G}(u,v)$. In particular, $d_{G}(u,u)=0$.
The transmission $\sigma (G)$ of $G$ is defined as the summation of distances between all unordered pairs of distinct vertices in $G$, i.e., $\sigma (G)=\sum_{\{u,v\}\subseteq V(G)} d_{G}(u,v)$.
This concept is also known as  the sum of the distances \cite{Pl} and the Wiener index \cite{DEG,Ho, HJ,NT,R, R1,Wi}. The average (or mean) distance of $G$ is just $\frac{2}{n(n-1)}\sigma(G)$  with $n=|V(G)|$ \cite{BG,Ch,DE,DG,He,KW,Mo}. Networks with small transmission have good properties, thus they are often desirable. In recent years, bounding the transmission has been a topic of study for various authors. Hypergraph theory has been used in chemistry,
see, e.g.~\cite{GKS,KS,KS2,Ko}. As indicated in \cite{KS}, the hypergraph model gives a higher accuracy
of molecular structure description: the higher the accuracy of the model, the greater the diversity of the behavior of
its invariants.


Tang and Deng \cite{TD} determined  the unique unicyclic graphs ($2$-uniform unicyclic hypergraphs)   with maximum and minimum transmissions, respectively.
The transmissions of a connected hypergraph was discussed in \cite{Ko}.
Sun et al.~\cite{SW} computed the transmissions of some
special $k$-uniform hypergraphs, and provided a  lower bound for Wiener
 index of a $k$-uniform hypergraph with given circumference.
Guo et al.~\cite{GZL} determined the unique $k$-uniform hypertrees with maximum, second maximum and third maximum transmissions, as well as the unique $k$-uniform  hypertrees with  minimum, second minimum and third minimum transmissions, respectively. We established  a relation between degree
distance and the transmissions and a  relation between the Gutman index and the transmission for uniform hypertrees, see \cite{GZ}.

In this article, we characterize the unique uniform unicyclic hypergraphs of fixed size  with minimum and maximum transmissions, respectively.

\section{Preliminaries}

Let $G$ be a connected hypergraph. Let $\sigma _{G}(u)=\sum_{v\in V(G)}d_G(u,v)$. Then $\sigma (G)=\frac{1}{2}\sum_{u\in V(G)} \sigma _{G}(u)$. For $A\subseteq V(G)$, let $\sigma _G(A)=\sum_{\{u,v\}\subseteq A}d_{G}(u,v)$.
For $A,B\subseteq V(G)$ with $A\cap B=\emptyset$, let $\sigma _G(A,B)=\sum_{a\in A,b\in B} d_G(a,b)$.

For $X\subseteq V(G)$ with $X\ne\emptyset$, let $G[X]$ be the subhypergraph of $G$ induced by $X$, i.e., $G[X]$ has vertex set $X$ and edge
set $\{e\subseteq X: e\in E(G)\}$. 
For $e\in E(G)$, let $G-e$ be the subhypergraph of $G$ obtained by deleting  $e$.
For $u\in V(G)$, let $G-u$ be the subhypergraph of $G$ obtained by deleting $u$ and all edges containing $u$.
A component of a hypergraph $G$ is a maximal connected subhypergraph of $G$.

Let $G$ be a $k$-uniform hypergraph with $u,v\in V(G)$ and $e_1,\dots,e_r\in E(G)$
such that $u\in e_i$, $v\notin e_i$ and $e'_i\notin E(G)$ for $1\leq i\leq r$, where $e'_i=(e_i\setminus \{u\})\cup\{v\}$.
Let $G'$ be the hypergraph with $V(G')=V(G)$ and  $E(G')=(E(G)\setminus \{e_1,\ldots, e_r\} )\cup\{e'_1,\ldots,e'_r\}$.
Then we say that $G'$ is obtained from $G$ by moving edges $e_1,\ldots,e_r$ from $u$ to $v$.

For a hypergraph $G$ with  $v\in V(G)$, let  $E_G(v)$ be the set of edges of $G$ containing $v$.
The degree of a vertex $v$ in $G$, denoted by  $d_G(v)$, is defined as
$d_G(v)=|E_G(v)|$.

A path $P=( v_0,e_1,v_1,\dots,v_{p-1},e_p, v_p)$ in a $k$-uniform hypergraph $G$ is called a pendant path at $v_0$, if $d_G(v_0)\ge 2$,  $d_G(v_i)=2$ for $1\leq i\leq p-1$, $d_G(v)=1$ for $v\in e_i\setminus\{v_{i-1}, v_i\}$ with $1\leq i\leq p$,
and $d_G(v_p)=1$. A pendant edge is a pendant path of length $1$.

If $P$ is a pendant path of a hypergraph $G$ at $u$, we say $G$ is obtained from $H$ by attaching a pendant path $P$ at $u$ with $H=G[V(G)\setminus(V(P)\setminus\{u\})]$. If $P$ is a pendant edge at $u$ in $G$, then we also say that $G$ is obtained from $H$ by attaching a pendant edge at $u$.

Let $G$ be a connected $k$-uniform hypergraph with $ |E(G)|\geq 1$. For $u\in V(G)$, and  positive integers $p$ and $q$, let  $G_{u}(p,q)$ be the $k$-uniform hypergraph obtained from $G$ by attaching two pendant paths of lengths $p$ and $q$ at $u$, respectively,  and  $G_{u}(p,0)$ the $k$-uniform hypergraph obtained from $G$ by attaching a pendant path of length $p$ at $u$.

\begin{Lemma} \label{max-moving1}\cite{GZL} Let $G$ be a connected $k$-uniform hypergraph with $|E(G)|\geq1$ and $u\in V(G)$.
For integers $p\geq q\geq 1$, $\sigma (G_u(p, q))< \sigma (G_u(p+1, q-1))$.
\end{Lemma}

Let $G$ be a connected $k$-uniform hypergraph with $u,v\in e\in E(G)$.  For positive integers $p$ and  $q$, let   $G_{u,v}(p,q)$ be the $k$-uniform hypergraph obtained from $G$ by attaching a pendant path of length $p$ at $u$ and a pendant path   of length $q$ at $v$, and $G_{u,v}(p,0)$ the $k$-uniform hypergraph obtained from $G$ by attaching a pendant path of length $p$ at $u$. Let $G_{u,v}(0,q)=G_{v,u}(q, 0)$.

\begin{Lemma} \label{max-moving2}\cite{GZL}
Let $G$ be a connected $k$-uniform hypergraph with $|E(G)|\geq 2$,  $u,v\in e \in E(G)$ and $d_G(u)=1$.
For integers $p\geq q\geq 1$, $\sigma (G_{u,v}(p, q))< \sigma (G_{u,v}(p+1, q-1))$.
\end{Lemma}


Let $G$ be a connected $k$-uniform hypergraph with $|E(G)|\geq 2$, and let $e=\{w_1,\dots, w_k\}$ be a pendant edge of $G$ at $w_k$. For $1\le i\le k-1$,
let $H_i$ be a connected $k$-uniform hypergraph  with $v_i\in V(H_i)$. Suppose that $G, H_1, \dots, H_{k-1}$ are vertex--disjoint.
For $0\le s\le k-1$, let $G_{e,s}(H_1,\dots, H_{k-1})$  be the hypergraph obtained by identifying $w_{i}$ of $G$ and $v_i$ of $H_i$ for each $i$ with $s+1\leq i\leq k-1$  and identifying  $w_{k}$ of $G$ and $v_i$ of $H_i$ for all $i$ with $1\leq i\leq s$.

\begin{Lemma} \label{cycle-min-moving1}
Suppose that $|E(H_j)|\geq 1$ for some $j$ with $1\leq j\leq k-1$.
Then $W (G_{e,0}(H_1,\dots,H_{k-1}))>\sigma (G_{e,s}(H_1,\dots,H_{k-1}))$ for $j\leq s \leq k-1$.
\end{Lemma}
\begin{proof}
Let $H=G_{e,0}(H_1,\dots,H_{k-1})$ and $H'=G_{e,s}(H_1,\dots,H_{k-1})$.
For  $1\leq i\leq k-1$, let $V_i=V(H_i)\setminus\{w_i\}$.
As we pass from $H$ to $H'$, for $1\leq i\leq s$, the distance between a vertex of $V_i$ and a vertex of $U=V(G)\setminus(e\setminus\{w_k\})$ is decreased by $1$, the distance between a vertex of $V_i$ and $w_i$ is increased by $1$, and the distance between any other vertex pair remains unchanged or is decreased. Thus
\begin{eqnarray*}
\sigma (H)-\sigma (H')
&\geq &\sum^s_{i=1} (\sigma _H(V_i, U)+ \sigma _H(V_i, \{w_i\}))\\
&&-\sum^s_{i=1}(\sigma _{H'}(V_i, U)+\sigma _{H'}(V_i, \{w_i\}))\\
&=&\sum^s_{i=1} (\sigma _H(V_i, U)-\sigma _{H'}(V_i, U))\\
&&+\sum^s_{i=1} (\sigma _H(V_i, \{w_i\}-\sigma _{H'}(V_i, \{w_i\}))\\
&=&\sum^s_{i=1}|V_i|(|V(G)|-k+1)-\sum^s_{i=1}|V_i|\\
&=&\sum^s_{i=1}|V_i|(|V(G)|-k)\\
&>&0,
\end{eqnarray*}
implying that $\sigma (H)>\sigma (H')$.
\end{proof}

For a $k$-uniform unicyclic hypergraph $G$  with $V(G)=\{v_1, \dots, v_n\}$, where $n=m(k-1)$, if $E(G)=\{e_1, \dots, $ $ e_m\}$, where $e_i=\{v_{(i-1)(k-1)+1}, \dots, v_{(i-1)(k-1)+k}\}$ for $i=1, \dots, m$, and $v_{(m-1)(k-1)+k}=v_1$, then we call $G$ a $k$-uniform loose cycle, denoted by $C_{n,k}$.

For a $k$-uniform hypertree $G$ with $n$ vertices, if all edges share a common vertex $v$, then we call $G$ a $k$-uniform hyperstar (with center $v$), denoted by $S_{n,k}$. In particular, $S_{1,k}$ consists of a single vertex.

Let $G$ be a connected $k$-uniform hypergraph with an induced subhypergraph $C_{g(k-1),k}$, where $k\geq3$ and $g\geq 2$. Label the vertices of $C_{g(k-1),k}$
as above   with $v_{(g-1)(k-1)+k}=v_1$. If $G$ is a $k$-uniform unicyclic hypergraph, then   $G-E(C_{g(k-1),k})$ consists of $g(k-1)$ components, denoted by $H_1,\dots, H_{g(k-1)}$ with $v_i\in V(H_i)$ for $i=1,\dots, g(k-1)$.  In this case, we denote $G$ by $C^k_{g(k-1)}\left(H_1,\dots, H_{g(k-1)}\right)$.
%
%
%
If $H_{(i-1)(k-1)+1}$ is a $k$-uniform hyperstar $S_{t_i,k}$ with center $v_{(i-1)(k-1)+1}$ for $i=1, \dots, g$, and $V(H_j)=\{v_j\}$ for $j=(i-1)(k-1)+2,\ldots,i(k-1)$, where $t_i=|V(H_i)|$, then
we denote $C^k_{g(k-1)}\left(H_1,\dots, H_{g(k-1)}\right)$ by $C^k_g(t_1,\dots,t_g)$.

\section{Uniform unicyclic hypergraphs with minimum and maximum  transmissions}

Let $G$ be a unicyclic graph of size $m$. Then $m\ge 3$. Let $CP_m$ be the unicyclic graph obtained by identifying a vertex of $C_{3,2}$ and an end vertex of $P_{2m-5,2}$.
If $m=3$, then $G\cong CP_m$, and if $m=4$, then $G\cong C_{4,2}$ or  $C^2_3(1,0,0)$ and it is easy to see that these two graphs have equal Wiener indices.
If $m\ge 5$,  then $\sigma (G)\ge m+2\left[{m\choose 2}-m\right]=m(m-2)$ with equality
if and only if the diameter of $G$ is at most $2$, i.e.,
$G\cong CP_m$ or $C_{5,2}$ for $m=5$, and  $CP_m$ for $m\ge 6$.

\begin{Theorem} \label{cycle-min} For $k\geq 3$, let $G$ be a $k$-uniform unicyclic hypergraph of size $m\geq 2$ with minimum transmission.
Then $G\cong C_{3k-3,k}$ if $m=3$, and $G\cong C^k_2(m-2,0)$ if $m=2$ or $m\ge 4$.
\end{Theorem}

\begin{proof} The order of $G$ is $n=(k-1)m$.
If  the cycle length of $G$ is $2$, then there are exactly two edges with exactly two vertices in common, and thus
\[
\sigma (G)\ge \left[m{k\choose 2}-1\right]\cdot 1+\left[{n\choose 2}-m{k\choose 2}+1\right]\cdot 2 =(k-1)m\left[(k-1)m-1-\frac{k}{2}\right]+1.
\]
with equality  if and only if  the diameter of $G$ is at most $2$, i.e., $G\cong C_2^k(m-2,0)$.

If the cycle length of $G$ is at least $3$, then any two edges have at most one vertex in common, and thus
\[
\sigma (G)\ge m{k\choose 2}\cdot 1+\left[{n\choose 2}-m{k\choose 2}\right]\cdot 2=(k-1)m\left[(k-1)m-1-\frac{k}{2}\right].
\]
with equality  if and only if  the diameter of $G$ is at most $2$, i.e.,  $G\cong C_{3k-3,k}$ and $m=3$.

Suppose that the cycle length of $G$ is at least $3$ and $G\not\cong C_{3k-3,k}$.
Then the diameter of $G$ is at least three. If the diameter of $G$ is at least four, then it is obvious that
$\sigma (G)>(k-1)m\left[(k-1)m-1-\frac{k}{2}\right]+1$. If the diameter of $G$ is three, then it is impossible that there is exactly one pair of vertices with distance three, and thus
$\sigma (G)>(k-1)m\left[(k-1)m-1-\frac{k}{2}\right]+1$.

Now we know that if $m=3$, then $\sigma (G)=(k-1)m\left[(k-1)m-1-\frac{k}{2}\right]$ with $G\cong C_{3k-3,k}$, and if $m=2$ or $m\ge 4$, then  $\sigma (G)=(k-1)m\left[(k-1)m-1-\frac{k}{2}\right]+1$ with $G\cong C_2^k(m-2,0)$.
\end{proof}

Note that, for $m\ge 5$,
$C^2_3(m-3,0,0)$
is the unique unicyclic graph of size $m$ with maximum transmission, see \cite{TD}. In the following, we determine the unique hypergraphs with maximum transmission in the set of $k$-uniform unicyclic hypergraphs of fixed size for $k\ge 3$.

\begin{Lemma} \label{cycle-max-moving1}
For $k\geq 3$ and $g\geq 3$, let $G=C^k_{g(k-1)}(H_1,\ldots,H_{g(k-1)})$. Let $G_1^*$ be the $k$-uniform hypergraph obtained from $G$ by moving $e_1$ from $v_k$ to $v_{(g-1)(k-1)+1}$, and  $G_2^*$  the $k$-uniform hypergraph obtained from $G$ by moving $e_g$ from $v_{(g-1)(k-1)+1}$ to $v_k$.
Then
$\sigma (G_1^*)>\sigma (G)$ or $\sigma (G_2^*)>\sigma (G)$.
\end{Lemma}

\begin{proof}
For $1\leq i\leq g(k-1)$, let $V_i=V(H_i)$.
Let $A_1=\bigcup_{i=k}^{\frac{g(k-1)}{2}}V_i$ and $A_2=\bigcup_{i= \frac{g(k-1)}{2}+2}^{(g-1)(k-1)+1}V_i$ if $g$ is even,
and $A_1=\bigcup_{i=k}^{\frac{(g-1)(k-1)}{2}+1}V_i$ and $A_2=\bigcup_{i= \frac{(g+1)(k-1)}{2}+1}^{(g-1)(k-1)+1}V_i$ if $g$ is odd.

Suppose that $|A_1|\geq |A_2|$.
As we pass from $G$ to $G_1^*$,
the distance between a vertex of $U_1=\bigcup_{i=2}^{k-1}V_i$ and a vertex of $A_1$ is increased by at least $1$,
the distance between a vertex of $U_1$ and a vertex of $A_2$ is decreased by $1$,
the distance between a vertex of $V_1$ and $V_k$ is increased by $g-2$,
and the distance between any other vertex
pair is increased or remains unchanged. Thus
\begin{eqnarray*}
\sigma _{G^*_1}(U_1,A_1)-\sigma _G(U_1,A_1)&\geq&|U_1||A_1|,\\
\sigma _{G^*_1}(U_1,A_2)-\sigma _G(U_1,A_2)&=&-|U_1||A_2|,\\
\sigma _{G^*_1}(V_1,V_k)-\sigma _G(V_1,V_k)&=&(g-2)|V_1||V_k|,
\end{eqnarray*}
and
\begin{eqnarray*}
\sigma (G_1^*)-\sigma (G)&\geq&\sigma _{G^*_1}(U_1,A_1)+ \sigma _{G^*_1}(U_1,A_2)+ \sigma _{G^*_1}(V_1,V_k) \\
&&-\left(\sigma _G(U_1,A_1)+ \sigma _G(U_1,A_2)+ \sigma _G(V_1,V_k)\right) \\
&\geq& |U_1||A_1|-|U_1||A_2|+(g-2)|V_1||V_k|\\
&=&|U_1|(|A_1|-|A_2|)+(g-2)|V_1||V_k|\\
&> &0.
\end{eqnarray*}
It follows that $\sigma (G_1^*)> \sigma (G)$.

Now suppose that $|A_1|< |A_2|$.
As we pass from $G$ to $G_2^*$,
the distance between a vertex of $U_2=\bigcup_{i=(g-1)(k-1)+2}^{g(k-1)}V_i$ and a vertex of $A_1$ is decreased by  $1$,
the distance between a vertex of $U_2$ and a vertex of $A_2$ is increased by at least $1$,
the distance between a vertex of $V_1$ and a vertex of $V_{(g-1)(k-1)+1}$ is increased by $g-2$,
and the distance between any other vertex
pair is increased or remains unchanged. Thus
\begin{eqnarray*}
\sigma _{G^*_2}(U_2,A_1)-\sigma _G(U_2,A_1)&=&-|U_2||A_1|,\\
\sigma _{G^*_2}(U_2,A_2)-\sigma _G(U_2,A_2)&\geq&|U_2||A_2|,\\
\sigma _{G^*_2}(V_1,V_{(g-1)(k-1)+1})-\sigma _G(V_1,V_{(g-1)(k-1)+1})&=&(g-2)|V_1||V_{(g-1)(k-1)+1}|,
\end{eqnarray*}
and
\begin{eqnarray*}
\sigma (G_2^*)-\sigma (G)
&\geq&\sigma _{G^*_2}(U_2,A_1)+ \sigma _{G^*_2}(U_2,A_2)+ \sigma _{G^*_2}(V_1,V_{(g-1)(k-1)+1})\\
&&-\left(\sigma _G(U_2,A_1)+ \sigma _G(U_2,A_2)+ \sigma _G(V_1,V_{(g-1)(k-1)+1})\right)  \\
&\geq&- |U_2||A_1|+|U_2||A_2|+(g-2)|V_1||V_{(g-1)(k-1)+1}|\\
&=&|U_2|(|A_2|-|A_1|)+(g-2)|V_1||V_{(g-1)(k-1)+1}|\\
&>&0.
\end{eqnarray*}
Therefore $\sigma (G_2^*)>\sigma (G)$.
\end{proof}

\begin{Lemma} \label{cycle-max-moving2}
For $k\geq 3$, let $G=C^k_{2(k-1)}\left(H_1,\dots, H_{2(k-1)}\right)$ and
let $G^*$ be the $k$-uniform hypergraph obtained from $G$ by moving all edges in $E_{H_k}(v_k)$ from $v_k$ to $v_2$.
Suppose that $\sum_{i=k+1}^{2(k-1)}|V(H_i)|\geq |V(H_2)|$.
Then
$\sigma (G^*)>\sigma (G)$.
\end{Lemma}
\begin{proof}
As we pass from $G$ to $G^*$,
the distance between a vertex of $U_1=V(H_k)\setminus\{v_k\}$ and a vertex of $U_2=\bigcup_{i=k+1}^{2(k-1)}V(H_i)$ is increased by $1$,
the distance between a vertex of $U_1$ and a vertex of $V(H_2)$ is decreased by $1$,
the distance between a vertex of $U_1$ and $v_k$ is increased by $1$,
and the distance between any other vertex
pair remains unchanged. Thus
\begin{eqnarray*}
\sigma _{G^*}(U_1,U_2)-\sigma _G(U_1,U_2)&=&|U_1||U_2|\geq |U_1||V(H_2)|,\\
\sigma _{G^*}(U_1,V(H_2))-\sigma _G(U_1,V(H_2))&=&-|U_1||V(H_2)|,\\
\sigma _{G^*}(U_1,\{v_k\})-\sigma _G(U_1,\{v_k\})&=&|U_1|,
\end{eqnarray*}
and
\begin{eqnarray*}
\sigma (G^*)-\sigma (G)
&=&\sigma _{G^*}(U_1,U_2)+ \sigma _{G^*}(U_1,V(H_2)+  \sigma _{G^*}(U_1,\{v_k\}) \\
&&-\left(\sigma _G(U_1,U_2)+ \sigma _G(U_1,V(H_2))+ \sigma _G(U_1,\{v_k\})\right)\\
&= & |U_1||U_2|-|U_1||V(H_2)|+|U_1|\\
&> &0.
\end{eqnarray*}
Therefore $\sigma (G^*)> \sigma (G)$.
\end{proof}

For $k\geq 3$, and $p\ge q\ge 1$,
let $\widetilde{C}_2^k(p,q)=G_{w_1, w_1'}(p,q)$, where $G=C_{2k-2,k}$ with edges $e=\{u,v,w_1,\ldots,w_{k-2}\}$ and $f=\{u,v,w'_1,\ldots,w'_{k-2}\}$. Let $\widetilde{C}_2^k(0,0)=C_{2k-2,k}$.

\begin{Lemma} \label{cycle-max-moving3}
For $k\geq 3$ and $m\geq 4$,
if $1\leq p\leq q-2$ and $p+q=m-2$, then $W\left(\widetilde{C}_2^k(p+1,q-1)\right)>W\left(\widetilde{C}_2^k(p,q)\right)$.
\end{Lemma}
\begin{proof} Let $G=\widetilde{C}_2^k(p,q)$. Let $e,f$ be the edges of the cycle in $G$ with $e\cap f=\{u,v\}$.
Let $(u_0,e_1,u_1,\ldots, u_{p-1},e_p,u_p)$ and $(v_0,f_1,v_1,\ldots, v_{q-1},f_q,v_q)$ be the pendant paths at $u_0$ and $v_0$, respectively, where $u_0\in e\setminus\{u,v\}$ and $v_0\in f\setminus\{u,v\}$.
Let $G'$ be the $k$-uniform hypergraph obtained from $G$ by moving $f_q$ from $v_{q-1}$ to $u_p$. Obviously, $G'\cong \widetilde{C}_2^k(p+1,q-1)$.
Let $U=V(G)\setminus(f_q\setminus\{v_{q-1}\})$.
By direct calculation, we have
\begin{eqnarray*}
&&\sigma _G(f_q\setminus\{v_{q-1}\},U)\\
&=&\sum_{w\in f_q\setminus\{v_{q-1}\}}\left(\sum^{q-1}_{i=0}d_G(w,v_i)+\sum^p_{i=1}d_G(w,u_i)+\sum^{q-1}_{i=1}\sum_{z\in f_i\setminus\{v_{i-1},v_i\}}d_G(w,z)\right.\\
&&\left.+\sum^{p}_{i=1}\sum_{z\in e_i\setminus\{u_{i-1},u_i\}}d_G(w,z)+\sum_{z\in f\setminus\{v_0\}}d_G(w,z)+\sum_{z\in e\setminus\{u,v\}}d_G(w,z)\right)\\
&=&\sum_{w\in f_q\setminus\{v_{q-1}\}}\left(\sum^{q-1}_{i=0}(q-i)+\sum^p_{i=1}(q+2+i)+\sum^{q-1}_{i=1}\sum_{z\in f_i\setminus\{v_{i-1},v_i\}}(q+1-i)\right.\\
&&\left.+\sum^{p}_{i=1}\sum_{z\in e_i\setminus\{u_{i-1},u_i\}}(q+2+i)+\sum_{z\in f\setminus\{v_0\}}(q+1)+\sum_{z\in e\setminus\{u,v\}}(q+2)\right)\\
&=&(k-1)\left[\frac{(q+1)q}{2}+\frac{(p+2q+5)p}{2}+\sum^{q-1}_{i=1}(k-2)(q+1-i)\right.\\
&&\left.+\sum^{p}_{i=1}(k-2)(q+2+i)+(k-1)(q+1)+(k-2)(q+2)\right]\\
&=&(k-1)\left[\frac{(q+1)q}{2}+\frac{(p+2q+5)p}{2}+\frac{(k-2)(q+2)(q-1)}{2}\right.\\
&&\left.+\frac{(k-2)(p+2q+5)p}{2}+(k-1)(q+1)+(k-2)(q+2)\right]\\
&=&(k-1)\left[\frac{(q+1)q}{2}+\frac{(p+2q+5)p}{2}+\frac{(k-2)(p^2+q^2+2pq+5p+q-2)}{2}\right.\\
&&\left. \vphantom{\frac{(k-2)(p^2+q^2+2pq+5p+q-2)}{2}}
+(2k-3)(q+1)+k-2\right],
\end{eqnarray*}
and
\begin{eqnarray*}
&&\sigma _{G'}(f_q\setminus\{v_{q-1}\},U)\\
&=&(k-1)\left[\frac{(p+2)(p+1)}{2}+\frac{(q+2p+6)(q-1)}{2}\right.\\
&&\left. +\frac{(k-2)(p^2+q^2+2pq+5q+p-6)}{2}+(2k-3)(p+2)+k-2\right].
\end{eqnarray*}
Note that
\begin{eqnarray*}
\sigma (G)&=&\sigma _G(f_q\setminus\{v_{q-1}\})+\sigma _G(U)+\sigma _G(f_q\setminus\{v_{q-1}\},U),\\
\sigma (G')&=&\sigma _{G'}(f_q\setminus\{v_{q-1}\})+\sigma _{G'}(U)+\sigma _{G'}(f_q\setminus\{v_{q-1}\},U).
\end{eqnarray*}
Since $\sigma _G(U)=\sigma _{G'}(U)$ and $\sigma _G(f_q\setminus\{v_{q-1}\})=\sigma _{G'}(f_q\setminus\{v_{q-1}\})$,
we have
\begin{eqnarray*}
&&\sigma (G')-\sigma (G)\\
&=&\sigma _{G'}(f_q\setminus\{v_{q-1}\})+\sigma _{G'}(U)+\sigma _{G'}(f_q\setminus\{v_{q-1}\},U)\\
&&-\left(\sigma _G(f_q\setminus\{v_{q-1}\})+\sigma _G(U)+\sigma _G(f_q\setminus\{v_{q-1}\},U)\right)\\
&=&\sigma _G(f_q\setminus\{v_{q-1}\},U)-\sigma _{G'}(f_q\setminus\{v_{q-1}\},U)\\
&=&
(k-1)\left[\frac{(p+2)(p+1)}{2}+\frac{(q+2p+6)(q-1)}{2}\right.\\
&&\left. +\frac{(k-2)(p^2+q^2+2pq+5q+p-6)}{2}+(2k-3)(p+2)+k-2\right]\\
&&-(k-1)\left[\frac{(q+1)q}{2}+\frac{(p+2q+5)p}{2}+\frac{(k-2)(p^2+q^2+2pq+5p+q-2)}{2}\right.\\
&&\left. \vphantom{\frac{(k-2)(p^2+q^2+2pq+5p+q-2)}{2}}+(2k-3)(q+1)+k-2\right]\\
&=&(k-1)(q-p-1)\\
&>&0.
\end{eqnarray*}
Thus $\sigma (G')>\sigma (G)$.
\end{proof}

\begin{Theorem} \label{cycle-max} For $k\geq 3$ and $m\geq 2$, let $G$ be a $k$-uniform unicyclic hypergraph of size $m$  with maximum transmission.
Then $G\cong \widetilde{C}^k_2\left(\lfloor\frac{m-2}{2}\rfloor,\lceil\frac{m-2}{2}\rceil \right)$.
\end{Theorem}

\begin{proof}
It is trivial if $m=2$. Suppose that $m\geq3$. Let $C$ be the unique cycle of $G$.
By Lemma~\ref{cycle-max-moving1}, the length of $C$ is $2$.
Let $e_1=\{v_{1}, \dots, v_{k}\}$ and $e_2=\{v_{k}, \dots, v_{2k-1}\}$ be the edges of $C$, where $v_{2k-1}=v_1$. Let $H_i$ be the component of $G-E(C)$ containing $v_i$, where $1\leq i \leq 2k-2$. Then $G\cong C^k_{2k-2}(H_1,\ldots,H_{2k-2})$.
By Lemma~\ref{cycle-max-moving2},  $|V(H_1)|=|V(H_k)|=1$.

\noindent  {\bf Claim.} For each $i$ with $1\leq i \leq 2(k-1)$ and $i\ne 1,k$, if $|V(H_i)|>1$, then $H_i$ is a pendant path at $v_i$.

Suppose that $d_{H_i}(v_i)=1$. Let $\Delta(H_i)$ be the maximum degree of $H_i$.


Suppose that $\Delta(H_i)\geq 3$. Note that for $w\in V(H_i)\setminus\{v_i\}$, $d_{H_i}(w)=d_G(w)$ and $d_{H_i}(w,v_i)=d_G(w,v_i)$.
Choose a vertex $v\in V(H_i)$ of degree at least $3$ such that $d_G(v,v_i)$ is as large as possible.
Let $F_1,\dots,F_{d_G(v)}$ be the vertex--disjoint subhypergraphs of $G-v$ with $\bigcup_{j=1}^{d_G(v)}V(F_j)=V(G)\setminus\{v\}$ such that $G[V(F_j)\cup\{v\}]$  is a  $k$-uniform hypertree for $2\leq j\leq d_G(v)$ and $G[V(F_1)\cup\{v\}]$ is a $k$-uniform unicyclic hypergraph containing $v_i$. Suppose that $G[V(F_j)\cup\{v\}]$  is not a pendant path at $v$ for some $j$ with $2\leq j\leq d_G(v)$.
Then there is at least one edge in $ E(G[V(F_j)\cup\{v\}])$ with at least three vertices of degree $2$. We choose such an edge $e=\{w_1,\ldots,w_k\}$ by requiring that
$d_G(v,w_1)$
is as large as possible, where $d_G(v,w_1)=d_G(v,w_l)-1$ for $2\leq l\leq k$.
Then
there are  two pendant paths at different vertices of $e$, say $P$ at $w_s$ and $Q$ at $w_t$, where $2\leq s<t\leq k$. Let $p\geq q\geq 1$ be the lengths of $P$ and $Q$, respectively.
Then $G\cong N_{w_s,w_t}(p,q)$, where $N=G[V(G)\setminus (V(P\cup Q)\setminus\{w_s,w_t\}) ]$.
Obviously,  $d_N(w_s)=d_N(w_t)=1$ and $G'=N_{w_s,w_t}(p+1,q-1)$ is a $k$-uniform unicyclic hypergraph of size $m$.
By Lemma~\ref{max-moving2}, we have $\sigma (G)<\sigma (G')$, a contradiction.
Thus $G[V(F_j)\cup\{v\}]$ is a pendant path at $v$ for $2\leq j\leq d_G(v)$.
Let $l_j$ be the length of the pendant path $G[V(F_j)\cup\{v\}]$ at $v$, where  $2\leq j\leq d_G(v)$ and $l_j\geq 1$.
Then $G\cong F_v(l_2,l_3)$, where $F=G[V(G)\setminus V(F_2\cup F_3)]$. We may assume that $l_2\ge l_3$. Obviously, $G''=F_v(l_2+1,l_3-1)$ is  a $k$-uniform unicyclic hypergraph of size $m$.
By Lemma~\ref{max-moving1}, $\sigma (G)<\sigma (G'')$, a contradiction.
Thus $\Delta(H_i)\le 2$.

If $\Delta(H_i)=1$, then $H_i$ is a pendant edge at $v_i$ in $G$, and the result follows. Suppose that $\Delta(H_i)=2$.
Suppose that $H_i$ is not a pendant path at $v_i$.
Then there is an edge in $H_i$ with at least three vertices of degree $2$.
Choose such an edge $e=\{w_1,\dots,w_k\}$ in $ E(H_i)$ such that $d_G(v_i,w_1)$ is as large as possible, where $d_G(v_i,w_1)=d_G(v_i,w_j)-1$ for $2\leq j\leq k$.
Then there are two pendant paths at different vertices of $e$, say $P'$ at $w_s$ and $Q'$ at $w_t$, where $2\leq s<t\leq k$.
Let $p'$ and $q'$ be the lengths of $P'$ and $Q'$, respectively, where $p'\geq q'\geq 1$.
Then $G\cong \widetilde{N}_{w_s,w_t}(p',q')$, where $\widetilde{N}=G[V(G)\setminus (V(P'\cup Q')\setminus \{w_s,w_t\}) ]$.
Obviously, $d_{\widetilde{N}}(w_s)=d_{\widetilde{N}}(w_t)=1$ and $G^*=\widetilde{N}_{w_s,w_t}(p'+1,q'-1)$ is  a $k$-uniform unicyclic hypergraph of size $m$.
By Lemma~\ref{max-moving2}, we have $\sigma (G)<\sigma (G^*)$, a contradiction.
Thus  $H_i$ is a pendant path at $v_i$ in $G$.

Next suppose that $d_{H_i}(v_i)\geq2$. Let $U_1,\ldots, U_{d_{H_i}(v_i)}$ be the subhypergraphs of $H_i-v_i$
with $\bigcup_{j=1}^{d_{H_i}(v_i)}V(U_j)=V(H_i)\setminus\{v_i\}$ such that $H_i[V(U_j)\cup\{v_i\}]$ is a $k$-uniform hypertree for $1\leq j\leq d_{H_i}(v_i)$.  Similarly as above, each $H_i[V(U_j)\cup\{v_i\}]$ is a pendant path for $1\leq j\leq d_{H_i}(v_i)$.
Let $s_j\geq 1$ be the length of $H_i[V(U_j)\cup\{v_i\}]$ with $1\leq j\leq d_{H_i}(v_i)$.
We may assume that $s_1\geq s_2$. Then $G\cong \widetilde{F}_{v_i}(s_1,s_2)$, where $\widetilde{F}=G[V(G)\setminus(V(U_1)\cup V(U_2))]$. Let $G^{**}\cong \widetilde{F}_{v_i}(s_1+1,s_2-1)$.
Obviously, $G^{**}$ is a $k$-uniform unicyclic hypergraph of size $m$.
By Lemma~\ref{max-moving1}, we have $\sigma (G)<\sigma (G^{**})$, a contradiction. 
This proves the Claim.

Let $r_i=|E(H_i)|$ for $1\leq i\leq 2k-2$. Obviously, $r_1=r_k=0$.

Suppose that 
there are integers $i$ and $j$ with $2\leq i<j\leq k-1$ such that $r_i\geq r_j\geq 1$.
By our Claim, we have $G=H_{v_i,v_j}(r_i,r_j)$, where $H=G[V(G)\setminus (V(H_i\cup H_j)\setminus \{v_i,v_j\})]$.
Let $G'=H_{v_i,v_j}(r_i+1,r_j-1)$. Obviously,  $G'$ is a $k$-uniform unicyclic hypergraph of size $m$.
By Lemma~\ref{max-moving2}, $\sigma (G)<\sigma (G')$, a contradiction.
Thus there is at most one integer $i$ with $2\leq i\leq k-1$ such that $r_i\geq 1$.
Similarly, there is at most one integer $i$ with $k+1\leq i\leq 2k-2$ such that $r_i\geq 1$.
Therefore $G\cong \widetilde{C}^k_2(r_1,r_2)$, where $r_1\geq r_2\geq 0$ and $r_1+r_2=m-2$.
By Lemma~\ref{cycle-max-moving3}, we have $G\cong\widetilde{C}^k_2\left(\left\lfloor\frac{m-2}{2}\rfloor,\lceil\frac{m-2}{2}\right\rceil\right)$.
\end{proof}


\end{document}